\documentclass[12pt]{article}
\usepackage[a4paper, margin=2.3cm]{geometry}
\usepackage{amssymb,amsthm, amsmath}
\usepackage{hyperref}
\usepackage{color}
\usepackage[normalem]{ulem}

\newtheorem{theorem}{Theorem}[section]

\newtheorem{lemma}[theorem]{Lemma}

\theoremstyle{definition}
\newtheorem{remark}[theorem]{Remark}

\newcommand{\R}{\mathbb R}
\newcommand{\PP}{\mathbb P}

\author{Alex Cohen \and Frank de Zeeuw}
\title{A Sylvester--Gallai theorem for cubic curves}

\begin{document}
\date{}
\maketitle

\begin{abstract}
We prove a variant of the Sylvester--Gallai theorem for cubics (algebraic curves of degree three):
If a finite set of sufficiently many points in $\R^2$ is not contained in a cubic, 
then there is a cubic that contains exactly nine of the points.
This resolves the first unknown case of a conjecture of Wiseman and Wilson from 1988, 
who proved a variant of Sylvester--Gallai for conics and conjectured that similar statements hold for curves of any degree.
\end{abstract}

\section{Introduction}

The Sylvester--Gallai theorem is a classic result in combinatorial geometry about points and line in the real plane.
For an overview of its background and history, see for instance \cite{GT}.

\begin{theorem}[Sylvester--Gallai]\label{thm:SG}
If $A$ is a finite set of points in $\R^2$ that is not contained in a line,
then there is a line that contains exactly two points of $A$.
\end{theorem}

This simple statement has inspired many variants.
Here we are interested in similar statements where lines are replaced by algebraic curves of higher degree. 
Elliott \cite{E} initiated the study of such questions by proving an analogue for circles: If a finite set of points in $\R^2$ is not contained in a line or a circle, then there is a circle that contains exactly three of the points. 
Note that three is the number of points that (typically) determine a circle.
In general, one expects a Sylvester--Gallai-type statement for a certain type of curve to deliver a curve containing exactly the number of points that would typically determine such a curve. 
Such a curve is called \emph{ordinary}.
Thus, an ordinary line is a line with two points from the given set, and an ordinary circle is a circle with three points.

Wiseman and Wilson \cite{WW} proved the following theorem on ordinary conics (where by \emph{conic} we mean an algebraic curves of degree two, not necessarily irreducible). Note that five is the number of points that (typically) determine a conic, so an ordinary conic contains five points from the given set.

\begin{theorem}[Wiseman--Wilson]\label{thm:WW}
If $A$ is a finite set of points in $\R^2$ that is not contained in a conic,
then there is a conic that contains exactly five points of $A$.
\end{theorem}

The proof of Wiseman and Wilson was somewhat lengthy and convoluted.
Czaplinski et al. \cite{Polish} provided a shorter proof using Cremona transformations.
Another short proof was given by Boys, Valculescu, and De Zeeuw \cite{BVZ}, based on Veronese mappings.

Wiseman and Wilson \cite{WW} conjectured that a similar statement holds for curves of any degree $d$; this is Problem 7.2.3 of Brass, Moser, and Pach \cite{BMP}.
Note that a curve of degree $d$ is typically determined by $d(d+3)/2$ points, so the statement should be as follows: If a finite set of points in $\R^2$ is not contained in a curve of degree $d$, then there is a curve of degree $d$ that contains exactly $d(d+3)/2$ of these points.
The papers \cite{Polish, BVZ} both mention this conjecture, but their techniques appear to be difficult to extend to curves of higher degree.

In this paper, we give yet another proof of Theorem \ref{thm:WW}, which turns out to be straightforward enough that we can extend it to curves of degree three, proving the Wiseman--Wilson conjecture for cubics (for sufficiently large sets). 
Our main theorem is the following statement.
To be precise, by a \emph{cubic} we mean an algebraic curve of degree three, i.e.
the zero set in $\R^2$ of a polynomial of degree three; we do not require the polynomial to be irreducible.

\begin{theorem}\label{thm:main}
If $A$ is a finite set of at least $250$ points in $\R^2$ that is not contained in a cubic,
then there is a cubic that contains exactly nine points of $A$.
\end{theorem}

The key idea in our proof is to work in a parameter space.
Roughly speaking, after fixing some of the given points, we move to the parameter space of curves passing through those fixed points, and in that space we use the Sylvester--Gallai theorem to find a curve that passes through the right number of further points.   
Specifically, in the case of cubics, we fix seven points, and consider the space of cubics passing through those points; 
if the seven points are carefully chosen, 
that parameter space should be two-dimensional. 
A point in this space corresponds to a cubic passing through the seven fixed points, 
and a point in the original plane corresponds to a line in this space. 
Thus we can use a dual version of the Sylvester--Gallai theorem to get a point in the parameter space that is contained in exactly two lines, which in the original plane corresponds to a cubic containing exactly two of the given points, as well as the seven fixed points,
making it an ordinary cubic.

Turning this outline into a proof is harder than it sounds. 
Most importantly, the fixed points have to be chosen so that the parameter space has the right dimension, and so that we can safely apply the dual Sylvester--Gallai theorem.
For conics this is not too hard, but for cubics it requires more work.
In higher degree the difficulties are greater,
and we have not been able to prove variants for curves beyond cubics; in fact, we are not entirely convinced that the conjecture holds in all degrees.

One shortcoming of our result is that it only holds for sufficiently large sets. 
This seems to be an artifact of our proof, and we expect that Theorem \ref{thm:main} holds for sets of any size. 
The number $250$ that results from our proof could probably be improved somewhat, but we were not able to lower it all the way to $10$.
It should also be noted that in our Theorem \ref{thm:main}, the cubic is not necessarily \emph{determined} by the nine points, i.e., it may not be the only cubic containing those nine points. In each of \cite{WW, Polish, BVZ}, the conic containing five points is also determined by those five points. Our proof does not seem to be able to give a cubic through nine points that is also determined by the nine points.

The structure of this paper is very straightforward: In Section $2$ we introduce the parameter space terminology and state some useful lemmas, in Section $3$ we give our new proof of Theorem \ref{thm:WW} on ordinary conics, and in Section $4$ we prove our main Theorem \ref{thm:main} on ordinary cubics.

\newpage

\section{Definitions and lemmas}

Throughout our proofs we will work in real projective space $\R\PP^2$.
For a vector space $V$, we write $\PP(V)$ to denote the projective space obtained from $V$, consisting of all lines through the origin in $V$.
We write $\R[x,y,z]_d$ for the ring of homogeneous polynomials of degree $d$ in three variables over $\R$.
This is a vector space of dimension $d(d+3)/2+1$.
For a finite point set $A \subset \R\PP^2$, we denote by $S^d_A$ the subspace of degree $d$ polynomials that vanish on $A$, i.e.
\[S^d_A = \{ f \in \R[x,y,z]_d : f(p)=0~\text{for all}~ p\in A \}. \]
We denote the associated projective space by
\[\mathcal{P}^d_A = \PP\left(S^d_A\right).\]
Then $\mathcal{P}^d_A$ is the \emph{parameter space} of curves of degree $d$ that contain $A$.
Indeed, a point in $\mathcal{P}^d_A$ is an equivalence class consisting of scalar multiples of a homogeneous polynomial in $S^d_A$, which defines a curve in $\R\PP^2$ that contains $A$.

We will frequently encounter the situation where $\mathcal{P}^d_A$ is a projective plane, 
and for some $x\not\in A$ the space $\mathcal{P}^d_{A\cup\{x\}}$ is a projective line.
Note that with a slight abuse of notation we can consider the line $\mathcal{P}^d_{A\cup\{x\}}$ as a subset of the plane $\mathcal{P}^d_A$; indeed, $S^d_{A\cup\{x\}}$ is a linear subspace of $S^d_A$, so the set of lines through the origin in  $S^d_{A\cup\{x\}}$ is naturally a subset of the set of lines through the origin in $S^d_A$.

The condition of vanishing at one point is a linear condition on the vector space of polynomials.
If the points in $A$ impose linearly independent conditions, then the subspace $S^d_A$ has dimension $d(d+3)/2+1 - |A|$, as long as this number is nonnegative,
and then the parameter space $\mathcal{P}^d_A$ has dimension $d(d+3)/2 - |A|$ as a projective variety.
Of course this is not the case for all point sets. 
For instance, if $A$ consists of five collinear points, then there are infinitely many conics containing $A$ (the line with five points combined with any other line), so $\mathcal{P}^2_A$ cannot have the expected dimension $0$.

The following basic fact (see Eisenbud, Green, and Harris \cite[Proposition 1]{EGH}) characterizes the small point sets whose parameter space does not have the expected dimension; 
in the words of \cite{EGH}, these point sets ``fail to impose independent conditions on curves of degree $d$".

\begin{lemma}\label{lem:eisenbud}
Let $A$ be a set of at most $2d+2$ points in $\R\PP^2$. 
The parameter space $\mathcal{P}^d_A$ fails to have dimension $d(d+3)/2 - |A|$  if and only if either $d+2$ of the points of $A$ are collinear or $|A| = 2d+2$ and $A$ is contained in a conic. 
\end{lemma}

We will also use two particular facts related to the Sylvester--Gallai theorem.
The first is a kind of dual version of Theorem \ref{thm:SG} in the sense of projective point-line duality. 

\begin{lemma}\label{lem:SGdual}
Let $\mathcal{L}$ be a finite set of lines in $\R\PP^2$, not all concurrent, and let $L$ be a line not in $\mathcal{L}$.
Then there is a point in $\R\PP^2$, not on $L$, that is contained in exactly two of the lines in $\mathcal{L}$.
\end{lemma}

Note that this statement is not exactly the dual of Theorem \ref{thm:SG},
because of the required avoidance of the line $L$.
See Lenchner \cite[Chapter 2]{L} for a proof and discussion of this statement.

The other fact we will use is the following statement proved by Boys, Valculescu, and De Zeeuw \cite[Lemma 2.1]{BVZ}.

\begin{lemma}\label{lines_few_points}
If $A$ is a finite set in $\R\PP^2$ that is not contained in a line, then there is a point $x \in A$ that is contained in two lines, each with exactly two or three points of $A$.
\end{lemma}

\newpage

\section{Ordinary conics revisited}

In this section we give our proof of Theorem \ref{thm:WW}, which will be the model for the proof of our main theorem on ordinary cubics in the next section.

For convenience we record the relevant consequences of Lemma \ref{lem:eisenbud}.

\begin{lemma}\label{lem:conicdimensions}
\verb% %\\
$(a)$ If $B\subset \R\PP^2$ consists of three points, 
then $\mathcal{P}^2_B$, the parameter space of conics containing $B$, has dimension $2$.\\ 
$(b)$ If $B\subset \R\PP^2$ consists of four points, not on one line, then $\mathcal{P}^2_B$ is has dimension $1$.\\
$(c)$ If $B\subset \R\PP^2$ consists of five points, 
no four on a line, then $\mathcal{P}^2_B$ has dimension $0$.
\end{lemma}

Given the finite point set $A\subset\R\PP^2$, we will choose a non-collinear three-point set $B\subset A$ in a specific way, and use it to construct a map to lines in the parameter space of conics containing $B$, as follows.
For a point $x \in \R\PP^2 \backslash B$, 
we consider $B\cup\{x\}$, which is a non-collinear four-point set.
By Lemma \ref{lem:conicdimensions}$(b)$,
$\mathcal{P}^2_{B\cup\{x\}}$ is a line, which we can consider as a line in the plane $\mathcal{P}^2_B$.
Thus we can define a map $\varphi^2_B$ that takes a point $x \in \R\PP^2 \backslash B$ to the line $\mathcal{P}^2_{B\cup\{x\}}$ in the plane $\mathcal{P}^2_B$;
i.e.
\[\varphi^2_B\left(x\right) = \mathcal{P}^2_{B\cup\{x\}}.\]

In our proof it will be crucial to know which lines in $\mathcal{P}^2_B$ have unique preimages under $\varphi^2_B$ and which have multiple preimages.
The following lemma shows that $\varphi^2_B$ is injective except on the lines spanned by $B$.

\begin{lemma}\label{lem:conicpreimages}
Let $B\subset \R\PP^2$ be a set of three non-collinear points. If $x, y \in \R\PP^2 \backslash B$ satisfy $\varphi^2_B(x) = \varphi^2_B(y)$ then $x$ and $y$ lie on a line through two points of $B$.
\end{lemma}
\begin{proof}
We have $\varphi^2_B(x) = \varphi^2_B(y)$ when $S^2_{B\cup\{x\}} = S^2_{B\cup\{y\}}$, which occurs precisely when the intersection $S^2_{B\cup\{x\}} \cap S^2_{B\cup\{y\}} = S^2_{B\cup\{x,y\}}$ is two-dimensional,
so $\mathcal{P}^2_{B\cup\{x,y\}}$ is one-dimensional. 
By Lemma \ref{lem:conicdimensions}$(c)$, this occurs only when $B\cup\{x,y\}$ has some four points collinear. 
\end{proof}

Now we show that $B$ can be chosen so that the image of $A\backslash B$ under $\varphi^2_B$ contains only one line with multiple preimages, while the other lines are not concurrent. This will allow us to apply the dual Sylvester--Gallai theorem to the set of lines $\varphi^2_B(A\backslash B)$.

\begin{lemma}\label{lem:goodBforconics}
Let $A \subset \R\PP^2$ be a finite set of points that is not contained in a conic. 
Then there is a three-point set $B \subset A$ such that $\varphi^2_B(A\backslash B)$ contains at most one line with multiple preimages in $A\backslash B$, and the remaining lines in $\varphi^2_B(A\backslash B)$ are not all concurrent.
\end{lemma}
\begin{proof}
By Lemma \ref{lines_few_points} we can choose a point $x_0\in A$ that lies on two lines $L_1,L_2$ that each contain two or three points of $A$.
Let $x_1$ be a point of $A$ on $L_1$,
$x_2$ a point of $A$ on $L_2$, 
and let $L$ be the line through $x_1$ and $x_2$. 
We choose $B = \{x_0,x_1,x_2\}$.
The lines $L_1,L_2$ each contain at most one point of $A\backslash B$, and $L$ can have any number of points from $A\backslash B$.
By Lemma \ref{lem:conicpreimages}, this implies that $\varphi^2_B(A\backslash B)$ has at most one line with multiple preimages; specifically, each point of $A\backslash B$ on $L$ is mapped to the same line, while every other point of $A\backslash B$ is mapped to a distinct line.

If $L$ contains at most one point of $A\backslash B$, 
then no line in $\varphi^2_B(A\backslash B)$ has multiple preimages in $A\backslash B$.
Then the lines in $\varphi^2_B(A\backslash B)$ could not be concurrent at some point in $\mathcal{P}^2_B$,
because that point would correspond to a conic that contains all of $A$.
Thus in this case $B$ is as required.

Suppose that $L$ contains more than one point of $A\backslash B$.
If the lines in $\varphi^2_B(A\backslash B)$ with a unique preimage are not all concurrent, then $B$ is as required.
Suppose the lines in $\varphi^2_B(A\backslash B)$ with a unique preimage are concurrent. 
Then the point of concurrency corresponds to a conic $C$ that contains $B$ and all the points of $A\backslash B$ outside $L$, 
which means that $A$ is contained in the cubic curve $L\cup C$.

First suppose that $C$ is irreducible.
We choose a new set 
$B' = \{y_0, y_1, y_2\}\subset A$ such that $y_0\in L\backslash C$ and $\{y_1,y_2\} \subset C \backslash L$. 
Then any line through two points of $B'$ contains at most three points of $A$, since $A$ lies on the cubic $L\cup C$, and a line intersects a cubic in at most three points, unless it is a component.
The latter is not possible, because that component would have to be $L$, and we chose $B'$ with only one point on $L$. 
So none of the lines in $\varphi^2_{B'}(A\backslash B')$ has multiple preimages, 
and they are not all concurrent since $A$ is not contained in a conic.
Thus $B'$ is as required.

Finally, suppose that $C$ is reducible, so it is a union of two lines $L_1$ and $L_2$ different from $L$.
We choose $B' = \{y_0, y_1, y_2\}$ again with $y_0\in L\backslash C$, 
but now with $y_1$ on $L_1\backslash (L\cup L_2)$ and $y_2$ on $L_2\backslash (L\cup L_1)$.
As above, a line through two points of $B'$ contains at most three points of $A$,
for the same reason, except that now we should also consider the possibility that the line through two points of $B'$ is a component of $C$, i.e. that it equals $L_1$ or $L_2$.
Because we chose $B'$ so that no two points of $B'$ lie on $L_1$ or on $L_2$, this possibility is excluded.
This completes the proof.
\end{proof}

With these preparations the proof of Theorem \ref{thm:WW} is straightforward.

\begin{proof}
Lemma \ref{lem:goodBforconics} gives us a three-point set $B \subset A$ such that $\varphi^2_B(A\backslash B)$ is a set of lines, out of which at most one line has multiple preimages in $A\backslash B$, while the remaining lines are not all concurrent. 
This allows us to apply Lemma \ref{lem:SGdual}, 
which gives us a point $z \in \mathcal{P}^2_B$ with exactly two lines $\varphi^2_B(x), \varphi^2_B(y) \in \varphi^2_B(A\backslash B)$ passing through it, neither of which have multiple preimages. 
The conic  corresponding to $z$ passes through $B\cup\{x,y\}$ and no other points of $A$, so it is an ordinary conic.
\end{proof}

\section{Ordinary cubics}\label{sec:ordinary_cubics}

We follow a similar outline to our proof for conics in the previous section.
First we record the relevant consequences of Lemma \ref{lem:eisenbud}. 

\begin{lemma}\label{lem:cubicdimensions}
\verb% %\\
$(a)$ If $B\subset \R\PP^2$ consists of seven points with at most four collinear, then $\mathcal{P}^3_B$, the parameter space of cubics containing $B$, has dimension $2$.\\
$(b)$ If $B\subset \R\PP^2$ consists of eight points with at most four collinear and at most seven co-conic, 
then $\mathcal{P}^3_B$ has dimension $1$.
\end{lemma}

We require one more statement of this kind, which does not follow from Lemma \ref{lem:eisenbud} and requires some more work.

\begin{lemma}\label{lem:tenpoints}
If $B\subset \R\PP^2$ is a set of ten points such that $\mathcal{P}^3_B$ has dimension $1$, 
then $B$ has either six points on a line or nine points on a conic.
\end{lemma}
\begin{proof}
If $\mathcal{P}^3_B$ is one-dimensional,
then $S^3_B$ is a two-dimensional vector space,
 so we can find two linearly independent polynomials in it,
 which gives us two distinct cubics containing $B$. 
Since they intersect in ten points, by B\'ezout's inequality they must share a component, as otherwise they could only intersect in nine points. 
That component may be a line or a conic, 
but either way we can conclude that $B$ lies on the union of a line $L$ and a conic $C$. 
We distinguish three cases.

First suppose that $C$ is irreducible and passes through at least seven points of $B$. 
Then any cubic through $B$ must contain all of $C$ by B\'ezout's inequality. Because there must be two distinct cubics containing $B$ there must be at least two lines passing through $B \backslash C$, and so there is only one point on $L$. Thus there are at least nine points on the conic $C$. 

Next suppose that $C$ is irreducible and passes through less than seven points of $B$. Then $L$ must contain at least four points, and so any cubic containing $B$ must contain all of $L$ by B\'ezout's inequality. But because there are at least two cubics containing $B$ there must be at least two conics containing $B \backslash L$; as we assumed $C$ is irreducible, it contains at most four points. So the line $L$ passes through at least six points.

Finally, suppose that $C$ is reducible. 
Because $C$ is a conic, it must be a union of two lines, so $B$ lies on three lines $L_1, L_2, L_3$. Suppose $|L_1 \cap B| \geq |L_2 \cap B| \geq |L_3\cap B|$. 
Then $|L_1\cap B| \geq 4$, so by B\'ezout's inequality any cubic passing through $B$ contains $L_1$ as a component. 
Since there are two distinct cubics containing $B$ there must be two distinct conics that contain $B\backslash L_1$.
This is only possible if one of the lines $L_2, L_3$ contains only one point of $B\backslash L_1$,
since otherwise $L_2\cup L_3$ would be the only conic containing $B \backslash L_1$.
Thus we have $|L_3 \cap B| \leq 1$, which implies $|L_1 \cap B| \geq 5$.
Then either $L_1$ is a line with at least six points,
or we have $|L_1 \cap B| = 5$ and $|L_2 \cap B| = 4$, which means that $L_1\cup L_2$ is a conic with nine points.
\end{proof}

We now define a map to the parameter space similar to that in the previous section.
Given a finite point set $A\subset \R\PP^2$,
we will find a seven-point set $B \subset A$ with at most three points collinear and at most six points co-conic. 
By Lemma \ref{lem:cubicdimensions}, 
$\mathcal{P}^3_B$ is then a plane,
and for any point $x\in \R\PP^2 \backslash B$,
the set $\mathcal{P}^3_{B\cup\{x\}}$ is a line in that plane.
We define a map $\varphi^3_B$ that takes a point $x \in \R\PP^2 \backslash B$ to the line $\mathcal{P}^3_{B\cup\{x\}}$ in the plane $\mathcal{P}^3_B$;
i.e.
\[\varphi^3_B\left(x\right) = \mathcal{P}^3_{B\cup\{x\}}.\]



The following lemma is an analogue of Lemma \ref{lem:conicpreimages}.
However, instead of lines with two preimages, here we analyze lines with three or more preimages.
The reason is that a line with exactly two preimages happens to easily give an ordinary cubic in our proof.
This allows us to focus on lines with three or more preimages, which are easier to characterize.

\begin{lemma}\label{lem:triplepreimage}
Let $B\subset \R\PP^2$ be a set of seven points with at most three points collinear and at most six points co-conic.
Let $x, y, z \in \R\PP^2 \backslash B$ be three points with $\varphi^3_B(x) = \varphi^3_B(y) = \varphi^3_B(z)$. 
Then either $x,y,z$ lie on a line through three points of $B$, 
or $x,y,z$ lie on a conic through six points of $B$.
\end{lemma}
\begin{proof}
By Lemma \ref{lem:cubicdimensions}, 
the assumptions on $B$ imply that $S_B$ is three-dimensional, and each of $S_{B\cup\{x\}}$, $S_{B\cup\{y\}}$, $S_{B\cup\{z\}}$ is two-dimensional.
The assumption that $\varphi^3_B(x) = \varphi^3_B(y) = \varphi^3_B(z)$ implies that 
$S_{B\cup\{x\}} = S_{B\cup\{y\}} = S_{B\cup\{z\}}$,
so $S_{B\cup\{x,y,z\}}$ is a two-dimensional vector space,
and $\mathcal{P}^3_{B\cup\{x,y,z\}}$ has dimension $1$.
Therefore, by Lemma \ref{lem:tenpoints}, $B\cup\{x,y,z\}$ has either a line passing through six points, or a conic passing through nine points. 
In the first case, the fact that $B$ has at most three on a line implies $x,y,z$ lie on a line through three points of $B$. 
In the second case, 
the fact that $B$ has at most six on a conic implies $x,y,z$ lie on a conic through six points of $B$.
Thus we have proved the lemma.
\end{proof}


The following lemma, an analogue of Lemma \ref{lem:goodBforconics}, gives us the set $B\subset A$ which has the properties that allow us to apply the dual Sylvester--Gallai theorem to the lines in $\varphi^3_B(A\backslash B)$.

\begin{lemma}\label{lem:goodBforcubics}
Let $A \subset \R\PP^2$ be a finite set of at least $250$ points, not all lying on a cubic. 
Then there is a seven-point set $B \subset A$ 
such that at most one line in $\varphi^3_B(A \backslash B)$ has more than two preimages in $A\backslash B$,
and the remaining lines of $\varphi^3_B(A\backslash B)$ are not concurrent.
\end{lemma}
\begin{proof}
We distinguish three cases, 
depending on whether or not there is a line or conic with many points. 
The case with many points on a line is separated into three subcases.

\bigskip
\noindent\textbf{Case 1:} There is no line with $14$ points of $A$ and no irreducible conic with $19$ points of $A$.

We will choose seven points $y_1,\ldots,y_7\in A$ so that no three are collinear and no six are co-conic. 
We choose $y_1,y_2$ arbitrarily.
We choose $y_3$ off the line through $y_1,y_2$, and we choose $y_4$ off of each line through two of $y_1,y_2,y_3$; 
this is possible since $A$ is not contained in a cubic. 
We choose $y_5$ off of each line through two of $y_1,y_2,y_3,y_4$; this is possible since each of these $6$ lines contains at most $13$ points, and $A$ has at least $250$ points. 

Since no three of $y_1,\ldots,y_5$ are collinear,
there is a unique conic through $y_1,\ldots,y_5$, and this conic is irreducible. 
Choose $y_6$ off of this conic and off of any line through two of $y_1,\ldots,y_5$; this is possible since the conic contains at most $18$ points and each line contains at most $13$ points.

Finally, we choose $y_7$ to avoid any line spanned by two of $y_1,\ldots,y_6$ and any conic spanned by five of $y_1,\ldots,y_6$; any such conic is irreducible.
There are $15$ such lines and $6$ such conics;
aside from the $6$ chosen points, each line contains at most $13-2=11$ more points, and each conic contains at most $18-5=13$ more points;
therefore, together these curves contain at most
$6 + 15\times 11 + 6\times 13=249$ points of $A$. 
Since $A$ contains at least $250$ points, 
we can choose $y_7$ so that 
$B = \{y_1,\ldots,y_7\}$ has no three points on a line and no six points on a conic. 
By Lemma \ref{lem:triplepreimage}, this implies that $\varphi^3_B(A\backslash B)$ has no lines with more than two preimages in $A\backslash B$.
The lines in $\varphi^3_B(A\backslash B)$ are not all concurrent since $A$ is not contained in a cubic.

\bigskip
\noindent \textbf{Case 2:} There is an irreducible conic $C \subset \R\PP^2$ passing through at least $19$ points of $A$.

We will choose $y_1,\ldots,y_7\in A$ with $y_1,y_2$ off of $C$ and $y_3,\ldots,y_7$ on $C$,
in such a way that no three are collinear and no six are co-conic.
We choose $y_1,y_2\in A\backslash C$ arbitrarily.
Then, for $i=3, \ldots, 7$, we choose $y_i \in A \cap C$ so that $y_i$ does not lie on any of the lines spanned by $y_1,\ldots, y_{i-1}$.
Moreover,
for $y_6$ we avoid the conic spanned by $y_1,\ldots, y_5$, 
and for $y_7$ we avoid the conics spanned by
any five of $y_1,\ldots, y_6$.
We set $B = \{y_1,\ldots,y_7\}$. 

We count how many points on $C$ need to be avoided when choosing $y_7$, aside from $y_3,\ldots,y_6$ (for the earlier choices the number is smaller).
For the line through $y_1$ and $y_2$ there are at most $2$ points on $C$ to avoid. For a line through two of $y_3,\ldots,y_6$ there is no further point to avoid.
For the $8$ lines through one of $y_1,y_2$ and one of $y_3,\ldots,y_6$, there is one more point on $C$ to avoid.
A conic containing $y_3,\ldots,y_6$ and one of $y_1,y_2$ contains no other points on $C$, since $C$ is irreducible.
The $4$ conics spanned by $y_1,y_2$ and three of $y_3,\ldots,y_6$ have at most one more point on $C$.
Altogether, 
there are at most $4+2+8 +4= 18$ points to avoid.
Since $C$ contains at least $19$ points of $A$, this is possible.

By construction $B$ has no three points collinear and no six points co-conic.
Therefore, by Lemma \ref{lem:triplepreimage}, $\varphi^3_B(A\backslash B)$ contains no line with more than two preimages, 
and the lines in $\varphi^3_B(A\backslash B)$ are not all concurrent since $A$ is not contained in a cubic.

\bigskip
\noindent\textbf{Case 3a:} There is a line $L \subset \R\PP^2$ passing through at least $14$ points of $A$,
and $|A\backslash L|\geq 10$.

We will choose seven points $y_1,\ldots,y_7\in A$ with $y_1,y_2,y_3,y_4$ off of $L$ and $y_5,y_6,y_7$ on $L$,
in such a way that no six are co-conic,
and no three are collinear except for $y_5,y_6,y_7$.
By applying Theorem \ref{thm:SG} to $A\backslash L$,
we find two points $y_1,y_2\in A\backslash L$ such that the line spanned by these two points contains no other points of $A \backslash L$.
We choose $y_3\in A\backslash L$ to be any other point off of $L$. Then there must be some point $y_4 \in A\backslash L$ not on a line through two of $y_1,y_2,y_3$; otherwise, $A \backslash L$ would be contained in the union of the lines through $y_1,y_3$ and $y_2,y_3$, so $A$ would be contained in a cubic.
Thus $y_1,y_2,y_3,y_4$ are in $A\backslash L$ and have no three collinear.

There are six lines spanned by $y_1,y_2,y_3,y_4$, which hit $L$ in at most six points; let $X_1$ be the set of these points. 
We choose $y_5$ to be any point in $A\cap L$ avoiding $X_1$.
Because no three of $y_1,\ldots,y_5$ are collinear, 
there is a unique conic through $y_1,\ldots,y_5$, which must be irreducible.
This conic intersects $L$ in at most two points
(one of which is $y_5$); let $X_2$ consist of these points.
We choose $y_6 \in A\cap L$ to avoid the at most $8$ points in $X_1\cup X_2$, so that $y_1,\ldots,y_6$ have no three collinear and are not all co-conic. 
Finally, we will pick $y_7\in L$ to avoid any conic through five of $y_1,\ldots,y_6$.
Note that a conic containing $y_5,y_6$ and any three of $y_1,y_2,y_3,y_4$ would not contain any more points on $L$, so for such a conic we do not have to avoid any points.
There is a unique conic through $y_1,y_2,y_3,y_4,y_6$,
and we let $X_3$ consist of the at most $2$ points where this conic intersects $L$ (including $y_6$).
Then we pick $y_7 \in A\cap L$ to avoid the at most $10$ points in $X_1\cup X_2\cup X_3$.
Then $B = \{y_1,\ldots,y_7\}$ has only one collinear triple and no six on a conic,
so Lemma \ref{lem:triplepreimage} tells us that $\varphi^3_B(A\backslash B)$ contains at most one line with more than two preimages.

It remains to show that we can select $B$ so that the lines of $\varphi^3_B(A\backslash B)$ with one or two preimages are not all concurrent. 
Suppose the lines with one or two preimages are all concurrent; then there is a cubic $D$ containing $A\backslash L$. 
There are at most three points of $L$ lying on the cubic $D$, so we may repeat the procedure  above for choosing $y_5,y_6,y_7$,
while also avoiding these three points.
This means that, when choosing $y_7$, we need to avoid at most $13$ points, which is possible since $L$ contains at least $14$ points of $A$.
After this rechoosing, the lines in $\varphi^3_B(A\backslash B)$ with one or two preimages are not all concurrent.
Indeed, otherwise there would be a cubic $E$ that  contains $A\backslash L$, 
just like $D$,
but that intersects $L$ in three different points.
Then $D$ and $E$ could not share any component,
because that would imply they have a common point on $L$.
Thus, by B\'ezout's inequality, they would intersect in at most $9$ points, contradicting the assumption that $A\backslash L$ has at least $10$ points.

\bigskip
\noindent \textbf{Case 3b:} There is a line $L \subset \R\PP^2$ passing through at least $14$ points of $A$,
we have $|A\backslash L|\leq 9$,
and there is a line $L'$ that contains at least three points of $A\backslash L$.

We will choose $y_1,\ldots,y_7\in A$ with $y_1,y_2,y_3$ outside $L\cup L'$, $y_4,y_5$ on $L'$, and $y_6,y_7$ on $L$.
Note that $|A\backslash L|\geq 6$, since otherwise $A\backslash L$ would be contained in a conic and $A$ would be contained in a cubic.
Similarly, $A\backslash (L\cup L')$ is not contained in a line,
so $A\backslash (L\cup L')$ consists of three to six points.

We choose any non-collinear $y_1, y_2,y_3\in A\backslash (L\cup L')$.
Because $|A\backslash (L\cup L')|\leq 6$,
there is at most one line that contains four or more points of $A\backslash (L\cup L')$,
and we choose $y_4,y_5\in A\cap L'$ to avoid this line, which is possible since $|A\cap L'|\geq 3$.
Then $y_1,\ldots,y_5$ have at most three collinear. 
Moreover, if a line passes through three of $y_1,\ldots,y_5$, then it passes through at most one more point of $A\backslash L$, and thus  it passes through at most two more points of $A$.

We choose $y_6\in L$ to avoid any line or conic determined by $y_1,\ldots,y_5$,
and we choose $y_7\in L$ to avoid any line or conic determined by $y_1,\ldots,y_6$. 
This is certainly possible, since in this case we have $|A\cap L|\geq 250-9 = 241$.
Then $B =\{y_1,\ldots,y_7\}$ has no six on a conic, at most three on a line, 
and any line through three passes through at most two more points of $A$.
Therefore, by Lemma \ref{lem:triplepreimage}, $\varphi^3_B(A\backslash B)$ contains no line with more than two preimages, 
and the lines in $\varphi^3_B(A\backslash B)$ are not all concurrent,
so $B$ is as required.

\bigskip
\noindent \textbf{Case 3c:} There is a line $L \subset \R\PP^2$ passing through at least $14$ points of $A$,
we have $|A\backslash L|\leq 9$,
and $A\backslash L$ has no three on a line.

We choose $y_1,\ldots,y_5\in A\backslash L$ arbitrarily,
and then we choose $y_6,y_7\in L$ so that $B = \{y_1,\ldots,y_7\}$ has no three collinear and no six co-conic.
Then $B$ is as required.
\end{proof}

\newpage
\begin{remark}
We calculated the number $250$ in this lemma with some care,
but we have not optimized it as much as possible.
For instance,
in Case 1 we could choose $y_1,y_2$ using Theorem \ref{thm:SG}, which would give one fewer line to avoid in the subsequent choices.
We have chosen not to make such adjustments, because they would complicate the process even more, 
while only giving small improvements.
We believe that with the current approach the condition on the size of $A$ could not be removed entirely.
\end{remark}

We are now ready to complete the proof of our main Theorem \ref{thm:main}.

\begin{proof}
Lemma \ref{lem:goodBforcubics} gives us a seven-point set $B \subset A$ such that $\varphi^3_B(A\backslash B)$ has at most one line with more than two preimages, and the lines with one or two preimages are not all concurrent.

If there is a line $L \in \varphi^3_B(A\backslash B)$ with exactly two preimages in $A\backslash B$, we are done. 
Indeed, for any point $z\in L$,
the cubic curve corresponding to $z$ passes through precisely nine points of $A$: 
the seven points of $B$ and the two points mapping to $L$. (Note that in this case, the resulting cubic is not uniquely determined, since $L$ cuts out a one dimensional pencil of cubics passing through these nine points of $A$)

If there is no point in $\varphi^3_B(A\backslash B)$ with exactly two preimages, 
then $\varphi^3_B(A\backslash B)$ is a set of lines, at most one of which has three or more preimages, while the other lines have unique preimages and are not all concurrent.
Thus we may apply Lemma \ref{lem:SGdual} to obtain a point $z \in \R\PP^2$, not on the line with multiple preimages, such that exactly two lines $\varphi^3_B(x), \varphi^3_B(y)\in \varphi^3_B(A\backslash B)$ pass through it, and both these lines have a unique preimage. 
Then the cubic curve corresponding to $z$ passes through $B \cup\{x,y\}$ and no other points of $A$, so it is an ordinary cubic.
\end{proof}

\begin{remark}
Recall that five points determine a unique conic if and only if no four of them are collinear. In other words, one can deduce uniqueness for degree two curves by looking at degree one relations. On the other hand, nine points might determine several cubic curves for more subtle algebraic reasons: one cannot determine uniqueness for degree three curves by looking at lower degree relations. For example, the nine intersection points of two irreducible cubics will generically have no three points collinear and no six points co-conic (see the Cayley-Bacharach theorem). This difference is why our theorem may deliver non uniquely determined cubics. Lemma 4.4 prunes lower degree algebraic relations, but cannot handle degree three algebraic relations. 
\end{remark}

\section{Conclusion}
Our main result shows that for any large enough set of points in $\R\PP^2$, not all lying on a cubic curve, there must exist a cubic curve passing through exactly nine points of the set. 
Our approach does not seem to apply to degree four and higher equations. 
We make key use of Lemma \ref{lem:eisenbud} to construct sets of seven points that determine a two dimensional space of cubics. This is possible because a set of seven points determines a higher dimensional space of cubics only if it possesses many points on a line or a conic. On the other hand, 12 points may determine a three dimensional space of quartics even if no three points are on a line, no six are on a conic, and no ten are on a cubic. For this reason, one cannot even define the analogue of the map $\varphi_B^4$ taking points in $A \setminus B$ to lines in the parameter space of quartics. 

Using Lemma 2.1, one could potentially construct a set of ten points $B\subset A$ (without too many on a line or too many on a conic) and then consider a well defined map from $A \setminus B$ to hyperplanes in the parameter space $\R\PP^4$. The problem of finding an ordinary quartic then reduces to a Sylvester-Gallai problem for points and hyperplanes in $\R\PP^4$. Given a set of hyperplanes in $\R\PP^4$ (each representing a parameter space of quartics through 11 points), one would like to find a point lying on exactly four hyperplanes (representing a quartic passing through exactly 14 points of the set $A$). In general there is no Sylvester-Gallai theorem avaiable for points and hyperplanes in $\R\PP^n$. For example, there are sets of points in $\R\PP^3$ such that no three determine an ordinary plane (consider e.g. points lying on two fixed lines). This obstruction does not rule out the approach, but suggests it may be difficult. It also suggests that the generalization of Wiseman and Wilson's theorem to degree four and higher curves may be false, just as the Sylvester-Gallai theorem is false in higher dimensions. A definitive answer either way would be very interesting.

\section*{Acknowledgements} 
The proofs in this paper were developed during the 2019 Combinatorics REU at Baruch College in New York. 
The REU was supported by
NSF awards DMS-1802059 and DMS-1851420.
We would like to thank Adam Sheffer and Pablo Sober\'on for organizing the REU and for all their support.

\end{document}